\documentclass[a4paper,11pt]{amsart}
\usepackage[all]{xy}
\usepackage{latexsym}
\usepackage{amssymb}
\usepackage{amsmath}

\newtheorem{theorem}{Theorem}
\newtheorem{proposition}[theorem]{Proposition}
\newtheorem{corollary}[theorem]{Corollary}
\newtheorem{lemma}[theorem]{Lemma}

\newcommand{\pb}{\ar@{}[dr]|{\text{\pigpenfont J}}}

\newcommand{\RMod}{R{\rm \mbox{-Mod}}}

\newcommand{\Rmod}{R{\rm \mbox{-mod}}}

\newcommand{\mcEproj}{\mcE{\rm \mbox{-}Proj}}

\newcommand{\mcEinj}{\mcE{\rm \mbox{-inj}}}

\newcommand{\Ab}{{\rm {\bf \mbox{Ab}}}}

\newcommand{\Add}{{\rm {\bf Add}}}

\newcommand{\Arr}{{\rm Arr}}

\newcommand{\Ext}{{\rm Ext}}
\newcommand{\LExt}{{\rm LExt}}

\newcommand{\Hom}{{\rm Hom}}

\newcommand{\im}{{\rm Im}}

\newcommand{\ME}{{\rm ME}}

\newcommand{\op}{{\rm {\footnotesize op}}}

\newcommand{\gre}{\epsilon}

\newcommand{\mcA}{\mathcal{A}}

\newcommand{\mcE}{\mathcal{E}}

\newcommand{\mcF}{\mathcal{F}}
\newcommand{\mcC}{\mathcal{C}}
\newcommand{\mcI}{\mathcal{I}}
\newcommand{\mcJ}{\mathcal{J}}
\newcommand{\mcK}{\mathcal{K}}

\newcommand{\mcM}{\mathcal{M}}

\newcommand{\dis}{\displaystyle}

\begin{document}

\footskip30pt

\date{}

\title{Lattice theoretic properties of approximating ideals}

\author{X.H.\ Fu}
\address{School of Mathematics and Statistics, Northeast Normal University, Changchun, CHINA}
\email{fuxianhui@gmail.com}

\author{I.\ Herzog}
\address{The Ohio State University at Lima, Lima, Ohio, USA}
\email{herzog.23@osu.edu}

\author{J.\ Hu}
\address{School of Mathematics and Physics, Jiangsu University of Technology,
 Changzhou, CHINA}
\email{jiangshenghu@jsut.edu.cn}

\author{H.\ Zhu}
\address{College of Science, Zhejiang University of Technology, Hangzhou, CHINA}
\email{hyzhu@zjut.edu.cn}

\thanks{This work was partially supported by NSFC (11671069,11771212), the Natural Science Foundation of Zhejiang Provincial (LY18A010032), the Qing Lan Project of Jiangsu Province, a Jiangsu Government Scholarship for Overseas Studies (JS-2019-328). The second author was partially supported by Northeast Normal University.}



\begin{abstract}
It is proved that a finite intersection of special preenveloping ideals in an exact category $(\mcA; \mcE)$ is a special preenveloping ideal. Dually, a finite intersection of special precovering ideals is a special precovering ideal. A counterexample of Happel and Unger shows that the analogous statement about special preenveloping subcategories does not hold in classical approximation theory. If the exact category has exact coproducts, resp., exact products, these results extend to intersections of infinite families of special peenveloping, resp., special precovering, ideals. These techniques yield the Bongartz-Eklof-Trlifaj Lemma:  if $a \colon A \to B$ is a morphism in $\mcA,$ then the ideal $a^{\perp}$ is special preenveloping. This is an ideal version of the Eklof-Trlifaj Lemma, but the proof is based on that of Bongartz' Lemma. The main consequence is that the ideal cotorsion pair generated by a small ideal is complete.
\end{abstract}

\maketitle

\section{Introduction}

Approximation theory~\cite{EJ, GT} is the part of relative homological algebra devoted to the question of which subcategories $\mcC$ of an abelian category $\mcA$ admit preenvelopes/precovers of objects $A \in \mcA.$  This question was first explored in special situations, by Enochs~\cite{E} in his work on flat precovers of modules and, independently, by Auslander and Smal\o~\cite{AS}, who were interested in the preprojective partition of an artin algebra; they called preenvelopes left approximations, and precovers right approximations.
The general theory is based on the more operable notion of a {\em special} $\mcC$-preenvelope, resp., precover. This is a preenvelope that appears as the monomorphism of a short exact sequence
$$\xymatrix@1@C=40pt{0 \ar[r] & A \ar[r]^e & \ar[r] C \ar[r] & Z \ar[r] & 0,}$$
where $C \in \mcC$ and $Z$ belongs to the left orthogonal category ${^{\perp}}\mcC := \{\,  X \; | \; \Ext^1 (X,C) = 0 \; \mbox{for all} \; C \in \mcC \, \}.$

{\em Ideal} approximation theory~\cite{FGHT} is a generalization of approximation theory in which the objects $A \in \mcA$ are approximated by the morphisms of an ideal $\mcJ$ of $\mcA.$ In order to define a {\em special} $\mcJ$-preenvelope, one assumes that the ambient category is endowed with an exact structure
$(\mcA; \mcE)$ and defines two morphisms $a : A_0 \to A_1$ and $b : B_0 \to B_1$ in $\mcA$ to be $\Ext${\em -orthogonal} if the map $\Ext (a,b),$ displayed as the diagonal in
$$\xymatrix@R=40pt@C=40pt{
\Ext (A_1, B_0)  \ar[r]^{\Ext (A_1, b)} \ar[d]_{\Ext(a,B_0)} \ar[rd]^{\Ext (a,b)} & \Ext (A_1, B_1) \ar[d]^{\Ext (a, B_1)} \\
\Ext (A_0, B_0) \ar[r]_{\Ext (A_0, b)} & \Ext (A_0, B_1),
}$$
is $0.$ A {\em special} $\mcJ${\em -preenvelope} of $B \in \mcA$ is a morphism $j : B \to C_0$ in $\mcJ$ that appears in a morphism of $\mcE$-conflations
\begin{equation} \label{Eq:special preenv}
\xymatrix@C=30pt@R=30pt{
B \ar[r]^j \ar@{=}[d] & C_0 \ar[r] \ar[d]^c & A_0 \ar[d]^a \\
B \ar[r] & C_1 \ar[r] & A_1,
}\end{equation}
with $a \in \mcI = {^{\perp}}\mcJ.$ In this paper, we prove the following main result.
\bigskip

\noindent {\bf Theorem~\ref{T:finite intersection}} {\em Let $\mcJ_1,$ $\mcJ_2 \lhd \mcA$ be ideals of $(\mcA; \mcE).$ If an object $B \in \mcA$ has a special $\mcJ_1$-preenvelope $m^1 : B \to C^1$ and a special $\mcJ_2$-preenvelope $m^2 : B \to C^2,$
then a special $(\mcJ_1 \cap \mcJ_2)$-special preenvelope of $B$ is given by the pushout $m^1 \amalg_B m^2 : B \to C^1 \amalg_B C^2.$}
\bigskip

The dual result, that the intersection of two special precovering ideals is also special precovering, is proved using the dual argument. There is no analogue of Theorem~\ref{T:finite intersection} in the classical theory, as indicated by a counterexample of Happel and Unger~\cite{HU}. Theorem~\ref{T:finite intersection} is closely related to the ideal version of Christensen's Lemma~\cite[Theorem 8.4]{FH}, which implies that the product of two special preenveloping ideals is also special preenveloping.
\bigskip

Salce defined a cotorsion pair in $\mcA$ to be a pair $(\mcF, \mcC)$ of subcategories such that $\mcF = {^{\perp}}\mcC$ and $\mcC = \mcF^{\perp}.$ The cornerstone of approximation theory is Salce's Lemma, which states that if $(\mcF, \mcC)$ is a cotorsion pair in $\mcA,$ then, under suitable conditions, every object in $\mcA$ has a special $\mcC$-preenvelope if and only if every object has a special $\mcF$-precover; such cotorsion pairs are called complete. An {\em ideal} cotorsion pair in an exact category $(\mcA; \mcE)$ is a pair $(\mcI, \mcJ)$ of ideals such that $\mcJ = \mcI^{\perp}$ and $\mcI = {^{\perp}}\mcJ.$  Ideal approximation theory began with an ideal version of Salce's Lemma~\cite[Theorem 18]{FGHT}, which states that if $(\mcI, \mcJ)$ is an ideal cotorsion pair in an exact category $(\mcA; \mcE)$ with enough projective morphisms and injective morphism, then every object in $\mcA$ admits a special $\mcJ$-preenvelope if and only if every object has a special $\mcI$-precover. Such cotorsion pairs $(\mcI, \mcJ)$ are called complete. Theorem~\ref{T:finite intersection} implies that if $(\mcI_1, \mcI_1^{\perp})$ and $(\mcI_2, \mcI_2^{\perp})$ are complete cotorsion pairs, then $\mcI_1^{\perp} \cap \mcI_2^{\perp}$ arises as the ideal on the right side of a complete cotorsion pair. The corresponding ideal on the left is given by the following. \bigskip

\noindent {\bf Corollary~\ref{C:finite sum}} {\em If $(\mcI_1, \mcI_1^{\perp})$ and $(\mcI_2, \mcI_2^{\perp})$ are complete cotorsion pairs in an exact category $(\mcA; \mcE)$ with enough projective morphisms and injective morphisms, then so is $(\mcEproj \diamond (\mcI_1 + \mcI_2), \mcI_1^{\perp} \cap \mcI_2^{\perp}).$}  \bigskip

It is a general fact from ideal approximation theory that whenever an ideal $\mcK \subseteq {^{\perp}}\mcM$ is contained in a left orthogonal ideal, then $\mcEproj \diamond \mcK \subseteq {^{\perp}}\mcM,$ so Corollary~\ref{C:finite sum} shows that the ideal $\mcI_1 + \mcI_2$ is almost a left orthogonal ideal, it just needs to be closed under ME extensions by projective morphisms (see Proposition \ref{P:left closed}). The infinite versions of Theorem~\ref{T:finite intersection} and Corollary~\ref{C:finite sum} are given in Theorem~\ref{T:infinite intersection} and Corollary~\ref{C:infinite sum}. A version of Theorem~\ref{T:infinite intersection} for projective classes in a triangulated category was proved by J.D.\ Christensen~\cite[Proposition 3.1]{C}.

In the last section, we prove an ideal version of the Eklof-Trlifaj Lemma~\cite[Theorem 2]{ET}, but the proof is based on the idea of the proof of Bongartz' Lemma~\cite{GT}, so we call this result the Bongartz-Eklof-Trlifaj (BET) Lemma. The Eklof-Trlifaj Lemma implies~\cite[Theorem 10]{ET} that a cotorsion pair generated by a set is complete. Analogously, the BET Lemma implies (Corollary~\ref{C:small left closure}) that an {\em ideal} cotorsion pair generated by a set is complete.
\pagebreak

\noindent {\bf Notation:} Let $A^i,$ $i \in I,$ be a family of objects in an additive category $\mcA$ and suppose that both the product $\Pi_i A^i$ and coproduct $\amalg_i A^i$ exist in $\mcA.$ If the canonical morphism from $\amalg_i A^i \to \Pi_i A^i$ is an isomorphism, we denote (see~\cite[p.83]{R}) the product/coproduct by $\oplus_i A^i.$ The {\em diagonal} morphism is denoted by $d = \left( \begin{array}{c } 1_A \\ 1_A \end{array} \right) \colon A \to A \oplus A,$ and the {\em sum} morphism by $s = (1_A, 1_A) \colon A \oplus A \to A.$

\section{Exact categories}

\subsection{Definition} In this paper, we work in the setting of an exact category $(\mcA; \mcE).$ For a comprehensive treatment of exact categories, we refer the reader to~\cite{B}, but we use the terminology of Keller~\cite{K}.
An exact category consists of an additive category $\mcA$ equipped with a collection $\mcE$ of distinguished kernel-cokernel pairs
$$\Xi \colon \xymatrix@1@C=45pt{B \ar[r]^m & C \ar[r]^p & A}$$
called {\em conflations.} The morphism $m$ is the {\em inflation of} $\Xi;$ the morphism $p$ the {\em deflation of} $\Xi.$ A morphism in $\mcA$ is called an {\em inflation,} resp., {\em deflation,} if there exists a conflation $\Xi$ of which it is the inflation, resp., deflation.  The collection $\mcE$ of conflations is closed under isomorphisms and satisfies the following axioms:

\begin{description}
\item[$E_0$] for every object $A \in \mcA,$ the morphism $1_A$ is an inflation, resp., a deflation;
\item[$E_1$] inflations, resp., deflations, are closed under composition;
\item[$E_2$] the pushout of an inflation along an arbitrary morphism exists and yields an inflation;
\item[$E_2^{\op}$] the pullback of a deflation along an arbitrary morphism exists and yields a deflation.
\end{description}

\noindent The conflations $\mcE$ of an exact category may also be considered as a category in the obvious way (see~\cite[Corollary 2.10]{B}). As such, $\mcE$ is an additive category by~\cite[Proposition 2.9]{B}.

\subsection{Ideals} An {\em ideal} $\mcJ$ of an additive category $\mcA,$ denoted $\mcJ \lhd \mcA,$ is an additive subbifunctor of $\Hom : \mcA^{\op} \times \mcA \to \Ab;$ it associates to a pair $(A,B)$ of morphisms in $\mcA$ a subgroup $\mcJ (A,B) \subseteq \Hom (A,B)$
and satisfies the usual conditions in the definition of an ideal in ring theory: for any morphism $g: A \to B$ in $\mcJ,$ the composition $fgh$ belongs to $\mcJ (X,Y),$ for any morphisms $f: B \to Y$ and $h: X \to A$ in $\mcA.$ If $\mcJ_1, \mcJ_2 \lhd \mcA$ are ideals, then the
intersection $\mcJ_1 \cap \mcJ_2 \lhd \mcA$ and sum $\mcJ_1 + \mcJ_2 \lhd \mcA$ are also ideals.

But if $\mcA$ is not a small category, then an ideal $\mcJ : (A,B) \mapsto \mcJ (A,B)$ is a class function, so in order to avoid the paradoxes of set theory, one must be careful about performing operations such as arbitrary intersection and sum of ideals, or even forming structures such as the lattice of all ideals.
For example, if $\{ a^i \; | \; i \in I \}$ is a class of morphisms in $\mcA,$ then the ideal $\langle a^i \; | \; i \in I \rangle \lhd \mcA$ generated by the $a^i$ is defined to be the ideal whose morphisms are the finite linear combinations
$\Sigma_i \, f_i a^i g_i.$ It is {\em not} defined as the intersection of all the ideals that contain the $a^i,$ for such a statement refers to ``the class of all ideals (classes)'' containing the $a_i.$ An ideal is {\em small} if it is generated by a {\em set} of morphisms.

\subsection{Approximating ideals} An ideal $\mcJ \lhd \mcA$ is {\em preenveloping} if for every object $B \in \mcA,$ there is a morphism $j : B \to J$ in $\mcJ$ with the property that every $j' : B \to J'$ in $\mcJ$ factors through $j,$
$$\xymatrix@C=30pt@R=30pt{B \ar[r]^j \ar[d]_{j'} & J \ar@{.>}[ld] \\
J'.}$$
The dual notion of a {\em precovering} ideal $\mcI \lhd \mcA$ is defined dually.

\begin{proposition} \label{P:preenv sum}
If $\mcJ_1,$ $\mcJ_2 \lhd \mcA$ are preenveloping, resp., precovering, ideals, then so is $\mcJ_1 + \mcJ_2.$
\end{proposition}

\begin{proof}
If $j^1 : B \to J^1$ is a $\mcJ^1$-preenvelope of $B$ and $j^2 : B \to J^2$ is a $\mcJ_2$-preenvelope, then the morphism ${\dis \left( \begin{array}{c }j^1 \\ j^2 \end{array} \right) \colon  \xymatrix@1{B \ar[r] & J^1 \oplus J^2}}$ is a $(\mcJ_1 + \mcJ_2)$-preenvelope. The dual result is proved dually.
\end{proof}

\subsection{Pushouts of inflations} Given inflations $m^1 : B \to C^1$ and $m^2 : B \to C^2,$ Axiom $E_1$ implies that the pushout $m = m^1 \amalg_B m^2 : B \to C^1 \amalg_B C^2$ is the inflation shown in the commutative diagram
$$\xymatrix@C=30pt@R=30pt{B \ar[r]^{m^1} \ar[d]_{m^2} \ar[rd]^{m} & C^1 \ar[r] \ar[d]^{p^1} & A^1 \ar@{=}[d] \\
C^2 \ar[r]^(.4){p^2} \ar[d] & C^1 \amalg_B C^2 \ar[r] \ar[d] & A^1 \\
A^2 \ar@{=}[r] & A^2,
}$$
where all the rows and columns are conflations. A preenveloping ideal $\mcJ \lhd \mcA$ is called {\em monic} preenveloping, if every object $B \in \mcA$ admits a $\mcJ$-preenvelope that is an inflation with respect to the exact structure $(\mcA; \mcE).$ The dual definition defines an {\em epic} precovering ideal.

\begin{proposition} \label{P:monic preenv sum}
If $\mcJ_1,$ $\mcJ_2 \lhd \mcA$ are monic preenveloping, resp.\ epic precovering ideals, then so are $\mcJ_1 \cap \mcJ_2$ and $\mcJ_1 + \mcJ_2.$
\end{proposition}

\begin{proof}
Suppose that $m^1 : B \to J_1$ and $m^2 : B \to J^2$ are inflations that serve as $\mcJ_1$-preenvelope and $\mcJ_2$-preenvelope, respectively. The pushout $m^1 \amalg_B m^2$ is an inflation that serves as a $(\mcJ_1 \cap \mcJ_2)$-preenvelope. 
To see that the $(\mcJ_1 + \mcJ_2)$-preenvelope of $B$ given in the proof of Proposition~\ref{P:preenv sum} is monic, express it as a composition
$$\left( \begin{array}{c }j^1 \\ j^2 \end{array} \right) \colon \xymatrix@C=45pt{B \ar[r]^(.4)d & B \oplus B \ar[r]^{m^1 \oplus m^2} & J^1 \oplus J^1.}$$
The first morphism is an inflation, by~\cite[Corollary 2.14]{B} (let all the objects equal $B$ and all the morphisms the identity $1_B$); the second by~\cite[Proposition 2.9]{B}. Then apply Axiom $E_1.$
\end{proof}

The next lemma, which is routine to verify, describes the cokernel of a pushout of inflations.

\begin{lemma}  \label{L:key}
Let $m^1 : B \to C^1$ and $m^2 : B \to C^2$ be inflations with a common domain. The pushout $m^1 {\small \amalg_B} \, m^2 : B \to C^1 \amalg_B C^2$ is the inflation that occurs in the pushout
$$\xymatrix@C=55pt@R=30pt{
B \oplus B \ar[r]^{m^1 \oplus m^2} \ar[d]^s & C^1 \oplus C^2  \ar[r] \ar[d]^{(p^1, p^2)} & A^1 \oplus A^2 \ar@{=}[d] \\
B \ar[r]^{m^1 {\small \amalg_B} \, m^2} & \; C^1 \amalg_B C^2 \ar[r] & \; A^1 \oplus A^2
}$$
of the direct sum of conflations along the sum morphism.
\end{lemma}

\section{The arrow category}

\subsection{Definition} The {\em arrow category} $\Arr (\mcA)$ of a category $\mcA$ is the category whose objects \linebreak $a: A_0 \to A_1$ are the morphisms (arrows) of $\mcA,$
and a morphism $f: a \to b$ in $\Arr (\mcA)$ is given by a pair of morphisms $f = (f_0, f_1)$ in $\mcA$ for which the diagram
\begin{equation} \label{Eq:arrow morphism}
\xymatrix@C=30pt@R=30pt{A_0 \ar[r]^{f_0} \ar[d]^{a} & B_0 \ar[d]^b \\
A_1 \ar[r]^{f_1} & B_1}
\end{equation}
commutes. The commutativity relation $bf_0 = f_1a$ implies, by Exercise 3.2.iii of~\cite{R}, that these pairs $(f_0, f_1) \in \Hom (A_0, B_0) \times \Hom (A_1, B_1)$ are the members of the pullback
\begin{equation} \label{Eq:Hom pullback}
\xymatrix@R=35pt{\ar@{}[dr]|(.2){\lrcorner}
\Hom_{\Arr (\mcA)} (a,b) \ar[r] \ar[d]  & \Hom (A_1, B_1) \ar[d]^{a^{\ast}} \\
\Hom (A_0, B_0) \ar[r]^{b_{\ast}} & \Hom (A_0, B_1),
}\end{equation}
where $a^{\ast}$ denotes pre-composition by $a$ and $b_{\ast}$ post-composition by $b.$ Every category $\mcA$ may be embedded, fully and faithfully, into its arrow category $\Arr (\mcA)$ by the functor $1: \mcA \to \Arr (\mcA),$ $B \mapsto 1_B,$ that associates to an object $B$ the identity arrow $1_B : B \to B.$

\subsection{The exact category of arrows} If $\mcA$ is an additive category, then so is $\Arr (\mcA).$ The direct sum of arrows $a: A_0 \to A_1$ and $b : B_0 \to B_1$ is given by $a \oplus b: \xymatrix@C=40pt{A_0 \oplus B_0 \ar[r]^{\left( \begin{array}{cc} a & 0 \\ 0 & b \end {array} \right)} & A_1 \oplus B_1.}$ An ideal $\mcI \lhd \mcA$ is closed under finite direct sums, as
$$a \oplus b = \left( \begin{array}{c} 1 \\ 0 \end{array} \right) a \; (1,0) + \left( \begin{array}{c} 0 \\ 1 \end{array} \right) b \; (0,1).$$
On the other hand, if $\mcI_1,$ $\mcI_2 \lhd \mcA,$ then $\mcI_1 + \mcI_2 = \langle a \oplus b \; |\; a \in \mcI_1, \; b \in \mcI_2 \rangle,$ because $$a + b : \xymatrix@1{A \ar[r]^(.4)d & A \oplus A \ar[r]^{a \oplus b} & B \oplus B \ar[r]^(.6)s & B.}$$

The additive category $\Arr (\mcA)$ may be equipped with an exact structure for which the collection $\Arr (\mcE)$ of distinguished kernel-cokernel pairs
$\xi: \xymatrix@1{b \ar[r]^m & c \ar[r]^p & a,}$ is given by
morphisms of conflations in $(\mcA; \mcE),$
\begin{equation} \label{Eq:arrow con}
\xymatrix@C=30pt@R=35pt{\Xi_0: B_0 \ar@<-2.5ex>[d]_{\xi} \ar@<2ex>[d]^b \ar[r]^-{m_0} & C_0 \ar[r]^{p_0} \ar[d]^c & A_0 \ar[d]^a \\
\Xi_1: B_1 \ar[r]^-{m_1} & C_1 \ar[r]^{p_1} & \; A_1.}
\end{equation}

The functor $\mcA \to \Arr (\mcA),$ $A \mapsto 1_A,$ is exact and so induces a morphism
$$\Ext_{(\mcA; \mcE)} (A,B) \mapsto \Ext_{(\Arr (\mcA); \Arr (\mcE))} (1_A, 1_B).$$ To see that this morphism of $\Ext$-groups is an isomorphism, consider a conflation of arrows of the form $\xi : \xymatrix@1@C=15pt{1_B \ar[r]^m & c \ar[r]^p & 1_A.}$
By~\cite[Corollary 3.2]{B}, $c: C_0 \to C_1$ is an isomorphism in $\mcA$ and $\xi$ is isomorphic to the extension $\xymatrix@1{1_B \ar[r]^{1_{m_0}} & 1_{C_0} \ar[r]^{1_{p_0}} & 1_A}$ that appears in the front of the diagram
$$\xymatrix@C=15pt@R=15pt{
 & B \ar[rr]^{m_0} \ar@{=}[dd]|!{[ld];[rd]}\hole  \ar@{=}[ld] && C_0 \ar[rr]^{p_0} \ar[dd]^(.7)c|!{[ld];[rd]}\hole \ar@{=}[ld] && A \ar@{=}[dd] \ar@{=}[ld] \\
B \ar[rr]^(.7){m_0} \ar@{=}[dd] && C_0 \ar[rr]^(.7){p_0} \ar@{=}[dd] && A \ar@{=}[dd] \\
& B \ar[rr]^(.3){m_1}|!{[ru];[rd]}\hole \ar@{=}[ld] && C_1 \ar[rr]^(.3){p_1}|!{[ru];[rd]}\hole \ar[ld]^(.4){c^{-1}} && A \ar@{=}[ld] \\
B \ar[rr]^(.7){m_0} && C_0 \ar[rr]^(.7){p_0} && A.
}$$

\subsection{$\Ext$-orthogonality} As defined in the Introduction, two morphisms $a$ and $b$ of $\Arr (\mcA)$ are $\Ext$-orthogonal if the morphism $\Ext_{\mcA} (a,b) : \Ext_{\mcA} (A_1, B_0) \to \Ext_{\mcA} (A_0, B_1)$ of abelian groups is $0.$
If $\mcM$ is a class of morphisms in $\mcA,$ then ${^{\perp}}\mcM$ is the ideal of morphisms left $\Ext$-orthogonal to every $m \in \mcM.$ An ideal of the form ${^{\perp}}\mcM$ is said to be {\em left closed,} and the {\em left closure} of a collection $\mcM$ is given by the left closed ideal
${\perp}(\mcM^{\perp}).$ One defines a {\em right closed} ideal, and the {\em right closure} of a collection of morphisms dually, and an ideal cotorsion pair consists of ideals $(\mcI, \mcJ)$ for which $\mcI = {^{\perp}}\mcJ$ and $\mcJ^{\perp} = \mcI.$ The ideal cotorsion pair of the form $({^{\perp}}(\mcM^{\perp}), \mcM^{\perp})$ is said to be {\em generated} by $\mcM,$ while the ideal cotorsion pair $({^{\perp}}\mcM, ({^{\perp}}\mcM)^{\perp})$ is {\em cogenerated} by $\mcM.$

\subsection{Special approximating ideals} The ideal $\mcJ$ is called {\em special preenveloping} if every object in $\mcA$ has a special $\mcJ$-preenvelope~(\ref{Eq:special preenv}); special precovering ideals are defined dually.
The special $\mcJ$-preenvelope $j : B \to C_0$ of an object $B \in \mcA$ displayed in~(\ref{Eq:special preenv}) may be viewed as part of a conflation in $(\Arr (\mcA), \Arr (\mcE)),$ given by $\xymatrix@1{1_B \ar[r] & c \ar[r] & a,}$ where $a \in {^{\perp}}\mcJ$ and the preenvelope is the domain morphism
$B \to C_0$ of the inflation.

\begin{theorem} \label{T:finite intersection}
If $\mcJ_1,$ $\mcJ_2 \lhd \mcA$ are special preenveloping ideals, then so is $\mcJ_1 \cap \mcJ_2.$ Moreover, if $m^1 : B \to C^1$ is a special $\mcJ_1$-preenvelope of $B,$ and $m^2 : B \to C^2$ is a special $\mcJ_2$-preenvelope of $B,$  then
a special $(\mcJ_1 \cap \mcJ_2)$-preenvelope of $B$ is given by the pushout $m^1 \amalg_B m^2 : B \mapsto C^1 \amalg_B C^2.$
\end{theorem}

\begin{proof}
Each special preenvelope may be viewed as the domain morphism of an inflation that is part of a conflation $\xymatrix@1{1_B \ar[r] & c^i \ar[r] & a^i,}$ $i = 1,2$ in the exact category $(\Arr (\mcA), \Arr (\mcE)),$ with $a^i \in {^{\perp}}\mcJ_i.$ By Lemma~\ref{L:key},
the pushout in $(\Arr (\mcA), \Arr (\mcE))$ is given by the inflation in
$$\xymatrix@1{1_B \ar[r] & c^1 \amalg c^2 \ar[r] & a^1 \oplus a^2.}$$
Now the domain morphism of this inflation is the pushout $m^1 \amalg_B m^2: B \to C^1 \amalg_B C^2$ of the two special preenvelopes, and so belongs to $\mcJ_1 \cap \mcJ_2.$ On the other hand, $a^1 \oplus a^2 \in {^{\perp}}\mcJ_1 + {^{\perp}}\mcJ_2 \subseteq {^{\perp}}(\mcJ_1 \cap \mcJ_2),$ as required.
\end{proof}

A dual argument shows that a finite intersection of special precovering ideals is itself special precovering. 

\subsection{A counterexample} Let $kQ$-mod be the category of finitely generated left $kQ$-modules where $k$ is an algebraically closed field and $kQ$ is the path algebra of the quiver $Q$
$$\xymatrix{
  &  2\ar[ld]  \\
1&              &  3. \ar[ll] \ar[lu]}$$
Let $T_1=P(1)\oplus P(3)\oplus\tau S(2)$ and $T_2=\tau T_1$. By~\cite[Example]{HU}, $T_1^\perp$ and $T_2^\perp$ are special preenveloping classes, but $T_1^\perp\bigcap T_2^\perp$ is not.

\section{The $\ME$-exact structure of arrows}

The goal of this section is to apply Salce's Lemma to Theorem~\ref{T:finite intersection} and describe the complete ideal cotorsion pair generated by a sum $\mcI_1 + \mcI_2$ of special precovering ideals.

\subsection{The pushout-pullback factorization} The conflation $\xi$ of arrows displayed in (\ref{Eq:arrow con}) admits a pushout-pullback factorization~\cite[Proposition 3.1]{B}
$$\xymatrix@C=30pt@R=30pt{\Xi_0 : \; B_0 \ar[r] \ar@<2ex>[d]^b \ar@<-3ex>[dd]_{\xi} & C_0 \ar[d]^{c_1} \ar[r] & A_0 \ar@{=}[d] \\
\;\;\;\Xi: B_1 \ar[r] \ar@{=}@<2ex>[d] & C \ar[r] \ar[d]^{c_2} & A_0 \ar[d]^a \\
\Xi_1: \; B_1 \ar[r] & C_1 \ar[r] & \; A_1,
}$$
which shows that the extension $\xi \in \Ext (a,b)$ is given by pairs $(\Xi_0, \Xi_1) \in \Ext (A_0, B_0) \times \Ext (A_1, B_1)$ in the pullback
\begin{equation} \label{Eq:Ext pullback}
\xymatrix@R=35pt@C=45pt{\ar@{}[dr]|(.2){\lrcorner}
\Ext_{\Arr (\mcA)} (a,b) \ar[r] \ar[d]  & \Ext_{\mcA} (A_1, B_1) \ar[d]^{\Ext (a, B_1)} \\
\Ext_{\mcA} (A_0, B_0) \ar[r]^{\Ext (A_0, b)} & \Ext_{\mcA} (A_0, B_1).
}\end{equation}
The composition $\gre : \Ext_{\Arr (\mcA)} (a,b) \to \Ext_{\mcA} (A_0, B_1)$ of the pullback diagram associates to $\xi$ the conflation $\Xi$ in the middle of the pushout-pullback decomposition.
A conflation \linebreak $\xi \in \Ext_{\Arr (\mcA)} (a,b)$ is {\em null-homotopic} if $\gre (\xi) = 0.$

\subsection{$\ME$-conflations of arrows} The most interesting conflations in $(\Arr (\mcA), \Arr (\mcE))$ are the $\ME$-conflations. A conflation of arrows $\xi: \xymatrix@1@C=15pt{b \ar[r] & c \ar[r] & a}$ in the exact structure $(\Arr (\mcA); \Arr (\mcE)$ is said to be $\ME$ ({\em mono-epi}) if there exists a factorization
$$\xymatrix{
\Xi_0 : B_0 \ar[r] \ar@{=}@<2ex>[d] \ar@<-3ex>[dd]_{\xi} & C_0 \ar[r] \ar[d]^{c_1} & A_0 \ar[d]^a \\
 \;\; \;\;\;\; B_0 \ar[r] \ar@<2ex>[d]^b & C \ar[r] \ar[d]^{c_2} & A_1 \ar@{=}[d]  \\
\Xi_1 : B_1 \ar[r] & C_1 \ar[r] & A_1,
}$$ with $c = c_2 c_1.$ The name ``mono-epi'' derives from the fact that in the homotopy category of conflations, which is abelian, this is a factorization of the morphism $\xi$ as a monomorphism followed by an epimorphism; the pushout-pullback factorization of
$\xi$ is the canonical epi-mono decomposition. The notation $c = b \star a$ is used to indicate that the arrow $c$ is an $\ME$-extension of $a$ by $b.$ \pagebreak

\subsection{Leibniz Ext} The pullback property of the diagram (\ref{Eq:Ext pullback}) ensures the existence of a unique morphism $\widehat{\Ext}(a,b)$ that makes the diagram
$$\xymatrix@R=40pt@C=40pt{
\Ext (A_1, B_0) \ar@/_/[ddr]_{\Ext (a,B_0)} \ar@/^/[rrd]^{\Ext (A_1, b)} \ar@{.>}[dr]|-{\widehat{\Ext} (a,b)} \\
 & \Ext_{\Arr (\mcA)} (a,b) \ar[r] \ar[d] \ar[rd]^{\gre} & \Ext (A_1, B_1) \ar[d]^{\Ext (a, B_1)} \\
 & \Ext (A_0, B_0) \ar[r]^{\Ext (A_0, b)} & \Ext (A_0, B_1)
}$$
commute. This morphism is an example of a {\em Leibniz cotensor}~\cite[Definition C.2.8]{RV} and, following the practice of Riehl and Verity, we call it the {\em Leibniz Ext} of the arrows $a$ and $b;$ its image \linebreak $\LExt (a,b) := \im \, \widehat{\Ext}(a,b) \subseteq \Ext_{\Arr (\mcA)} (a,b)$ consists of the 
$\ME$-conflations. If we denote by $\ME \subseteq \Arr (\mcE)$ the collection of $\ME$-conflations of arrows, then $(\Arr (\mcA); \ME)$ is an exact substructure of $(\Arr (\mcA); \Arr (\mcE))$ \cite[Theorem 3.2]{FH}, whose $\Ext$ functor is $\LExt.$ 
Because $\Ext (a,b) = \gre \, \widehat{\Ext}(a,b),$ the morphisms $a$ and $b$ are $\Ext$-orthogonal if and only if every $\ME$-extension $b \star a$ is null-homotopic.

In order to better understand Leibniz $\Ext,$ let us compare it to the {\em Leibniz Hom} of arrows $a$ and $b.$ This is the map $\widehat{\Hom} (a,b)\colon \Hom_{\mcA} (A_1, B_0) \to \Hom_{\Arr {\mcA}} (a,b)$ that exists by the pullback property of the diagram~(\ref{Eq:Hom pullback}),
$$\xymatrix@R=40pt@C=40pt{
\Hom (A_1, B_0) \ar@/_/[ddr]_{\Hom (a,B_0)} \ar@/^/[rrd]^{\Hom (A_1, b)} \ar@{.>}[dr]|-{\widehat{\Hom} (a,b)} \\
 & \Hom_{\Arr (\mcA)} (a,b) \ar[r] \ar[d] & \Hom (A_1, B_1) \ar[d]^{a^{\ast}} \\
 & \Hom (A_0, B_0) \ar[r]^{b_{\ast}} & \Hom (A_0, B_1).
}$$
Its image consists of the morphisms $f \colon a  \to b$ whose diagram (\ref{Eq:arrow morphism}) admits a morphism $v : A_1 \to B_0$ that makes both of the triangles in
$$\xymatrix@C=35pt@R=35pt{A_0 \ar[r]^{f_0} \ar[d]^{a} & B_0 \ar[d]^b \\
A_1 \ar[r]^{f_1} \ar@{.>}[ru]|-v & B_1}$$
commute. According to Definition 1.1.2 of~\cite{H}, the arrow $a$ has the {\em left lifting property} with respect to $b,$ or, $b$ has the {\em right lifting property} with respect to $a,$ if $\widehat{\Hom} (a,b)$ is onto. 
We can therefore define $a$ to have the {\em left $\Ext$-lifting property} with respect to $b,$ or $b$ to have the {\em right $\Ext$-lifting property} with respect to $a$ if every conflation of arrows in $\Ext_{\Arr (\mcA)} (a,b)$ is $\ME.$

\subsection{Projective morphisms}
A morphism $p: Z \to A$ in $(\mcA; \mcE)$ is {\em projective} if it belongs to the minimum left closed ideal $\mcEproj := {^{\perp}}(0^{\perp}) = {^{\perp}}\Hom,$ that is, if it is left $\Ext$-orthogonal to every morphism in $\mcA.$ Equivalently, the pullback along $p$ of any $\mcE$-conflation
$$\xymatrix@R=30pt@C=30pt{
 &  & Z \ar[d]^p \ar@{.>}[ld] \\
B \ar[r] & C \ar[r] & A
}$$
yields a trivial $\mcE$-conflation and a factorization, as displayed, occurs. A morphism $p: P_0 \to P_1$ in $(\mcA; \mcE)$ is projective if and only if every conflation $\xymatrix@1{a \ar[r] & b \ar[r] & p}$ in the $\ME$-exact category $(\Arr (\mcA), \ME)$ is null homotopic. The dual notion is that of an {\em injective} morphism, and the ideal of such is denoted by $\mcEinj.$

The following property of a left closed ideal is verified in~\cite[Corollary 5.4]{FH}.

\begin{proposition} \label{P:left closed}
If $\mcI \lhd (\mcA; \mcE)$ is a left closed ideal and $i \in \mcI,$ then  $p \star i \in \mcI,$ for every projective morphism $p.$
\end{proposition}

The proposition implies that for any ideal $\mcI \lhd \mcA,$ $\mcEproj \diamond \mcI \subseteq {^{\perp}}(\mcI^{\perp}).$ The category $(\mcA; \mcE)$ {\em has enough projective morphisms} if every object is the codomain of a projective morphism; {\em having enough injective morphisms} is defined dually.

\subsection{Salce's Lemma} Salce's Lemma is the cornerstone of approximation theory. The ideal version of Salce's Lemma, proved in~(\cite[Theorem 18]{FGHT} and \cite[Theorem 6.3]{FH}), states that if the category $(\mcA; \mcE)$ has enough projective and enough injective morphisms,
and $(\mcI, \mcJ)$ is a cotorsion pair in an exact category $(\mcA; \mcE)$ with enough projective and enough injective morphisms, then $\mcI$ is special precovering if and only if $\mcJ$ is special preenveloping. Such cotorsion pairs $(\mcI, \mcJ)$ are called {\em complete.}

Suppose that $\mcJ$ is a special precovering ideal and consider a special $\mcJ$- preenvelope $j : B \to C_0$ as displayed in~(\ref{Eq:special preenv}). The morphism $a :  A_0 \to A_1$ is called a $\mcJ${\em -cosyzygy} of $B,$ and a $\mcJ${\em -cosyzygy} ideal is an ideal contained in ${^{\perp}}\mcJ$ that contains for every $B \in \mcA$ a $\mcJ$-cosyzygy of $B.$ The proof of Theorem~\ref{T:finite intersection} indicates that if $\mcJ_1$ and $\mcJ_2$ are special preenveloping ideals, then ${^{\perp}}\mcJ_1 + {^{\perp}}\mcJ_2$ is $\mcJ_1 \cap \mcJ_2$-cosyzygy ideal. Now the dual of Theorem 6.4 of~\cite{FH} implies that if $\mcK$ is a $\mcJ$-cosyzygy ideal, then $\mcEproj \diamond \mcK = {^{\perp}}\mcJ,$ whence the following corollary.

\begin{corollary} \label{C:finite sum}
If $(\mcI_1, \mcI_1^{\perp})$ and $(\mcI_2, \mcI_2^{\perp})$ are complete cotorsion pairs in an exact category $(\mcA; \mcE)$ with enough projective morphisms and enough injective morphisms, then so is $$(\mcEproj \diamond (\mcI_1 + \mcI_2), \mcI_1^{\perp} \cap \mcI_2^{\perp}).$$
\end{corollary}

The hypotheses of Corollary~\ref{C:finite sum} are self dual, so if $({^{\perp}}\mcJ_1, \mcJ_1)$ and $({^{\perp}}\mcJ_2, \mcJ_2)$ are complete cotorsion pairs in $(\mcA; \mcE),$ then so is $({^{\perp}}\mcJ_1 \cap {^{\perp}}\mcJ_2, (\mcJ_1 + \mcJ_2) \diamond \mcEinj).$ We do not know if the sum of special precovering ideals, (resp., special preenveloping ideals), is already special precovering (resp., special preenveloping).

\section{The Infinite case}

In this section we make explicit some of the consequences of Grothendieck's Axiom AB4. In the context of exact categories, this axiom is equivalent is the existence of pushouts of arbitrary families of inflations.

\subsection{Definition} An exact category $(\mcA; \mcE)$ is said to {\em has exact coproducts} if coproducts exist in $\mcA$ and are exact. This means that if $\Xi^i \colon \xymatrix@1{B^i \ar[r]^{m^i} & C^i \ar[r]^{p^i} & A^i,}$ $i \in I,$ is a set of conflations in $(\mcA; \mcE),$ then the coproduct
$$\amalg_i \; \Xi^i \colon \xymatrix@1@C=45pt{\amalg_i \, B^i \ar[r]^{{\small \amalg_i} m^i} & \amalg_i \, A^i \ar[r]^{{\small \amalg_i} p^i} & \amalg_i \, C^i}$$
is an inflation in $\mcE.$ A coproduct of $I$-many copies of an object $B$ is denoted by $B^{(I)}.$ The following contains the key idea behind what follows; the proof is the same as that of Lemma~\ref{L:key}.

\begin{lemma}  \label{L:infinite pushouts}
Let $m^i : B \to C^i,$ $i \in I,$ be a set of inflations with common domain $B.$ If $(\mcA; \mcE)$ has exact coproducts, then the pushout $m = {\small \amalg_B} \; m^i  \colon B \to \amalg_B \, C^i$
exists in $(\mcA; \mcE)$ and is given by the inflation that occurs in the pushout
$$\xymatrix@C=55pt@R=30pt{
B^{(I)} \ar[r]^{{\small \amalg_i} m^i} \ar[d]^s & \amalg_i \, C^i   \ar[r] \ar[d] & \amalg_i \, A^i \ar@{=}[d] \\
B \ar[r]^{{\small \amalg_B} \, m^i} & \; \amalg_B \, C^i \ar[r] & \; \amalg_i \, A^i
}$$
of the coproduct of conflations along the sum morphism.
\end{lemma}

Because the domain conflation in the pushout of Lemma~\ref{L:infinite pushouts} is a coproduct, this morphism of conflations is determined by the restrictions to its components,
$$\xymatrix{B \ar[r]^{m^i} \ar[d]^(.4){b^i} & C^i \ar[r]^{p^i} \ar[d]^(.4){c^i} & A^i \ar[d]^(.4){a^i} \\
\; B^{(I)} \ar[r]^{\amalg \, m^i} \ar[d]^s & \amalg_i \, C^i \ar[r]^{\amalg \, p^i} \ar[d] & \amalg_i \, A^i \ar@{=}[d] \\
B \ar[r]^(.4)m & \amalg_B \, C^i \ar[r] & \amalg_i \, A^i,
}$$
where $b^i \colon B^i \to \amalg_i \, B^i,$ $c^i \colon C^i \to \amalg_i \, C^i$ and $a^i \colon A^i \to \amalg_i \, A^i$ are the structural inflations of the coproducts. Notice that $sb^i = 1_B,$ so that composition yields the pullback
$$\xymatrix{B \ar[r]^{m^i} \ar@{=}[d] & C^i \ar[r]^{p^i} \ar[d] & A^i \ar[d]^(.4){a^i} \\
B \ar[r]^(.4)m & \amalg_B \, C^i \ar[r] & \amalg_i \, A^i.
}$$

\subsection{The canonical morphism} Let $A^i,$ $i \in I,$ and $B$ be objects in an exact category $(\mcA; \mcE).$ Define the canonical morphism of abelian groups
\begin{equation} \label{Eq:canonical morphism}
\xymatrix@1@C=30pt{\Ext (\amalg_i \, A^i, B) \ar[r] & \prod_i \Ext (A^i,B),} \;\; \xi \mapsto (\xi^i)_{i \in I},
\end{equation}
so that the $i$-component is obtained by pullback along the structural inflation $a^i : A^i \to \oplus_i \, A^i,$
$$\xymatrix@R=30pt@C=30pt{\xi^i : B \ar[r]^-{m^i} \ar@{=}@<2ex>[d] & C^i \ar[r]^{d^i} \ar[d] & A^i \ar[d]^{a^i} \\
\xi : B \ar[r]^m & C \ar[r] & \amalg_i \, A^i.
}$$

\begin{proposition} \label{P:canonical morphism}
If $(\mcA; \mcE)$ has exact coproducts, then the canonical morphism~{\rm (\ref{Eq:canonical morphism})} is an isomorphism.
\end{proposition}

\begin{proof}
By Lemma~\ref{L:infinite pushouts}, the inverse of this map is defined by associating to $(\xi^i)_{i \in I}$ the conflation whose inflation is the pushout of the $m^i,$ $i \in I.$
\end{proof}

\subsection{Coproducts in the $\ME$-exact category} The canonical morphism~(\ref{Eq:canonical morphism}) is defined by taking pullbacks, so if we are working in the category $(\Arr (\mcA); \Arr (\mcE)),$ it respects $\ME$-conflations of arrows,
$$\xymatrix@C=30pt{\Ext_{\Arr (\mcA)} (\amalg_i \, a^i, b) \ar[r] & \prod_i \Ext_{\Arr (\mcA)} (a^i, b)  \\
\LExt_{\mcA} (\amalg_i \, a^i, b) \ar[r] \ar@{^{(}->}[u] & \prod_i \LExt_{\mcA} (a^i, b). \ar@{^{(}->}[u]
}$$
It $(\mcA; \mcE)$ has exact coproducts, then so does $(\Arr (\mcA); \Arr (\mcE))$ and and the functor $A \mapsto 1_A$ respects coproducts. Proposition~\ref{P:canonical morphism} then implies that the canonical morphism on top is an isomorphism. According to the following proposition, the $\ME$-exact category of arrows
also has exact coproducts, so that the canonical morphism on the bottom is also an isomorphism.

\begin{proposition}   \label{P:ME coproducts}
If $(\mcA; \mcE)$ has exact coproducts, then a coproduct of $\ME$-conflations in \linebreak $(\Arr (\mcA); \Arr (\mcE))$ is an $\ME$-conflation.
\end{proposition}

\begin{proof}
Let $\xi^i: \xymatrix@1@C=15pt{b^i \ar[r] & c^i \ar[r] & a^i,}$ $i \in I,$ be a set of $\ME$-conflations, with respective $\ME$ factorizations
$$\xymatrix@C=30pt@R=30pt{
B^i_0 \ar[r] \ar@{=}[d] & C^i_0 \ar[r] \ar[d]^{c^i_1} & A^i_0 \ar[d]^{a^i} \\
B^i_0 \ar[r] \ar[d]^{b^i} & C^i \ar[r] \ar[d]^{c^i_2} & A^i_1 \ar@{=}[d]  \\
B^i_1 \ar[r] & C^i_1 \ar[r] & A^i_1.
}$$
Then the coproduct $\amalg_i \, \xi^i: \xymatrix@1@C=15pt{\amalg_i \, b^i \ar[r] & \amalg_i \, c^i \ar[r] & \amalg_i \, a^i}$ has the $\ME$ factorization
$$\xymatrix@C=30pt@R=30pt{
\amalg_i \, B^i_0 \ar[r] \ar@{=}[d] & \amalg_i \, C^i_0 \ar[r] \ar[d]^{ \amalg_i c^i_1} &  \amalg_i \, A^i_0 \ar[d]^{ \amalg_i a^i} \\
 \amalg_i \, B^i_0 \ar[r] \ar[d]^{ \amalg_i b^i} &  \amalg_i \, C^i \ar[r] \ar[d]^{ \amalg_i c^i_2} &  \amalg_i \,  A^i_1 \ar@{=}[d]  \\
 \amalg_i \, B^i_1 \ar[r] &  \amalg_i \, C^i_1 \ar[r] &  \amalg_i \, A^i_1.
}$$
\end{proof}

\subsection{Left closed ideals} In this subsection, we show that the left closure of a set of morphisms is the left closure of the single morphism given by the coproduct of these morphisms.

\begin{lemma} \label{L:left closed}
A left closed ideal $\mcI$ in an exact category $(\mcA; \mcE)$ with exact coproducts is closed under coproducts of morphisms.
\end{lemma}

\begin{proof}
It must be shown that if $\Ext (a^j, b) = 0$ for all $j \in J,$ then $\Ext (\amalg_j \, a^j, b) = 0.$ We have that $b : B_0 \to B_1$ and $a^j : A^j_0 \to A^j_1,$ so that
$\Ext (\amalg_j \, a^j, b): \Ext (\amalg_j \, A^j_1, B_0) \to \Ext (\amalg_j \, A^j_0, B_1).$ Now use the canonical isomorphism~(\ref{Eq:canonical morphism}) to see that $\Ext (\amalg_j \, a^j, b)$ is the unique simultaneous solution for the horizontal arrow in the middle of all the commutative diagrams
$$\xymatrix@C=30pt@R=30pt{
\Ext (\amalg_j \, A^j_1, B_0) \ar[rr]^{\Ext (\amalg_j \, a^j, b)} \ar[d] && \Ext (\amalg_j \, A^j_0, B_1) \ar[d] \\
\prod_j \, \Ext (A^j_1, B_0) \ar[rr] \ar[d]^{\pi^j} && \prod_j \Ext (A^j_0, B_1) \ar[d]^{\pi^j} \\
\Ext (A^j_1, B_0) \ar[rr]^{\Ext (a^j, b)} && \Ext (A^j_0, B_1),
}$$
with $j \in J.$
\end{proof}

The proof of Lemma~\ref{L:left closed} implies that if $\mcI$ is a small ideal, generated by a set of morphisms $m^i,$ $i \in I,$ then $\mcI^{\perp} \subseteq m^{\perp},$ where $m = \amalg_i \, m^i.$ On the other hand, we showed earlier
that $\mcI \subseteq \langle m \rangle,$ so that $m^{\perp} \subseteq \mcI^{\perp}.$ The equality $m^{\perp} = \mcI^{\perp}$ implies that the left closure of $\mcI$ is given by the left closure ${^{\perp}}(m^{\perp})$ of $m.$

\subsection{Intersections of special preenveloping ideals} To generalize Theorem~\ref{T:finite intersection} to the infinite case, we have to refer to the intersection of a set $\mcJ_i,$ $i \in I,$ of special preenveloping ideals. As mentioned earlier, these ideals are, in general, proper classes, so in order for the intersection to be a class, there must be some uniform expression of the indexing $i \mapsto \mcJ_i.$ This is an underlying assumption that we make. We also freely abuse the language to refer to a set of ideals, when we really mean a collection of ideals uniformly indexed by a set.

\begin{theorem} \label{T:infinite intersection}
Let $\{ \mcJ_i \: | \:  i \in I \}$ be a set of special preenveloping ideals in an exact category $(\mcA; \mcE)$ with exact coproducts. Then the intersection ideal $\mcJ = \bigcap_{i \in I} \; \mcJ_i$ is also special preenveloping.
\end{theorem}

\begin{proof}
Let $B \in \mcA$ and suppose that a special $\mcJ_i$-preenvelope $\xymatrix@1@C=20pt{1_B \ar[r]^{m^i} & c^i \ar[r] & a^i,}$ $i \in I,$ is given. This means that the domain morphism $m^i_0 : B \to C_0^i$ belongs to $\mcJ_i$ and that $a_i \in {^{\perp}}\mcJ_i \subseteq {^{\perp}}(\bigcap_{i \in I} \; \mcJ_i).$  By Lemma~\ref{L:infinite pushouts} applied to the arrow category $(\Arr (\mcA); \Arr (\mcE)),$ the pushout of the inflations $m^i$ exists and is the inflation of the bottom conflation of
$$\xymatrix@C=30pt@R=30pt{
1_B \ar[r]^{m^i} \ar@{=}[d] & c^i \ar[r]  \ar[d]^(.45){h^i} & a^i \ar[d] \\
1_B \ar[r]^(.4)m & \amalg_B \, c^i \ar[r] & \oplus_i \, a^i.
}$$
The domain morphism of $m$ is a composition $m_0 = h_0^i m_0^i \in \mcJ_i,$ so that $m \in \bigcap_{i \in I} \; \mcJ_i,$ and $\amalg_i a^i \in{^{\perp}}(\bigcap_{i \in I} \; \mcJ_i),$ by Lemma~\ref{L:left closed}.
\end{proof}

Of course, if $(\mcA; \mcE)$ has exact products, then one can argue dually and use the existence of infinite pullbacks to prove that if $\mcI^j,$ $j \in J,$ is a set of special precovering ideals, then so is the intersection $\bigcap_j \mcI^j.$ This is analogous to the result~\cite[Proposition 3.1]{C} of J.D.\ Christensen that projective classes in a triangulated category are closed under intersection. 

\begin{corollary} \label{C:infinite sum}
Let $(\mcA; \mcE)$ be an exact category with exact coproducts, and enough injective morphism and projective morphisms. If $(\mcI_j, \mcI^{\perp}_j),$ $j \in J,$ is a set of complete cotorsion pairs in $(\mcA; \mcE),$ then the cotorsion pair $({^{\perp}}\mcJ, \mcJ),$ where $\mcJ = \bigcap_{j \in J} \; \mcI_j^{\perp},$ is complete. Moreover, ${^{\perp}}\mcJ$ is the left closure of $\sum_j \, \mcI_j.$
\end{corollary}

\begin{proof}
The first claim follows from Theorem~\ref{T:infinite intersection} and Salce's Lemma. To verify the second claim, note that ${^{\perp}}\mcJ$ is a left closed ideal, containing all the $\mcI_j,$ so it must contain the left closure $\mcI$ of $\sum_j \, \mcI_j.$ On the other hand, $\mcI \subseteq {^{\perp}}\mcJ$ is closed under coproducts, and is therefore a cosysygy ideal of $\mcJ.$ By Proposition~\ref{P:left closed} and~\cite[Theorem 6.4]{FH}, $\mcI = \mcEproj \diamond \mcI = {^{\perp}}\mcJ.$
\end{proof}

\section{The Bongartz-Eklof-Trlifaj Lemma}

The following result is an ideal version of the celebrated Eklof-Trlifaj Lemma~\cite[Theorem 2]{ET}. Its proof, however uses the argument from Bongartz' Lemma~\cite{GT}, where it is required that certain $\Ext$-groups be sets; this is not included in the definition of an exact category. We assume in this section that the exact category
$(\mcA; \mcE)$ has exact coproducts.

\begin{theorem} \label{T:Bongartz} {\rm (The Bongartz-Eklof-Trlifaj (BET) Lemma)}
Let $(\mcA; \mcE)$ be an exact category and suppose that an object $B \in \mcA$ and morphism $a : A_0 \to A_1$ are given, with the property that the $\Ext$-group $I = \Ext_{\Arr (\mcA)} (a, 1_B)$ is a set. A special $a^{\perp}$-preenvelope of $B$ is given by the pushout
$$\xymatrix@C=30pt@R=30pt{
1_B \ar[r]^{m^{\zeta}} \ar@{=}[d] & c^{\zeta} \ar[r]  \ar[d] & a \ar[d] \\
1_B \ar[r]^(.4)m & \amalg_B \, c^{\zeta} \ar[r] & a^{(I)},
}$$
indexed by $\zeta \in I.$
\end{theorem}

\begin{proof}
Lemma~\ref{L:left closed} implies that the coproduct $a^{(I)}$ belongs to the left closure of $a,$ so it remains to verify that the domain morphism $m_0$ of the pushout inflation
$$\xymatrix@C=30pt@R=30pt{
B \ar[r]^{m_0} \ar@{=}[d] & C_0 \ar[d]^{\amalg_B c^{\zeta}} \\
B \ar[r]^{m_1}  & C_1
}$$
belongs to $a^{\perp}.$ Recall that $\Ext_{\mcA} (a, m_0) = 0$ if and only if every $\ME$-conflation $\xymatrix@C=15pt{m_0 \ar[r] & y \ar[r] & a}$ is null homotopic, so assume that such an $\ME$-conflation is given, with $\ME$ factorization
$$\xymatrix@C=30pt@R=30pt{
B \ar@{=}[d] \ar[r] & Y_0 \ar[r] \ar[d]^{y_1} & A_0 \ar[d]^a \\
B \ar[r] \ar[d]^{m_0} & Y \ar[r] \ar[d]^{y_2} & A_1 \ar@{=}[d] \\
C_0 \ar[r] & Y_1 \ar[r] & A_1.
}$$
The morphism of conflations in the top half of the diagram is isomorphic to some $\xi \in \Ext_{\Arr (\mcA)} (a, 1_B),$ so we can replace $y_1 : Y_0 \to Y$ with some $c^{\zeta} : C_0^{\zeta} \to C_1^{\zeta},$ $\zeta \in I,$
$$\xymatrix@C=20pt@R=15pt{
& B \ar[ld]_(.4){m_0} \ar@{=}[dd]|!{[rr];[ld]}\hole \ar[rr]^{m_0^{\xi}} && C^{\xi}_0 \ar[rr] \ar[dd]^{c^{\xi}} \ar@{.>}[llld]^(.4){f_0} && A_0 \ar[dd]^a \\
C_0  \ar[dd]_{\amalg_B c^{\zeta}} \\
& B \ar[ld]_(.35){m_1} \ar[rr]^{m_1^{\xi}}  && C^{\xi}_1 \ar[rr] \ar@{.>}[llld]^{f_1} && A_1  \\
C_1.
}$$
The universal property of the pushout yields an extension of $m : 1_B \to \amalg_B \, c^{\zeta}$ along $m^{\xi},$ as indicated. The morphism $m_0 : B \to C_0$ therefore factors through $m_0^{\xi},$ which implies that the given $\ME$-conflation is null-homotopic.
\end{proof}

\subsection{The left closure of a set} The BET Lemma implies that the ideal $a^{\perp}$ is special preenveloping, so that Salce's Lemma yields the following.

\begin{corollary} \label{C:cyclic left closure}
Suppose that $(\mcA; \mcE)$ has enough projective morphisms and injective morphisms, and that $a \colon A_0 \to A_1$ is a morphism in $\mcA.$ The ideal cotorsion pair $({^{\perp}}(a^{\perp}), a^{\perp})$ generated by $a$ is complete. Furthermore, the closure ${^{\perp}}(a^{\perp})$ of $a$ is the least ideal containing $a$ and closed under coproducts of morphisms, and $\ME$-extensions by projective morphisms.
\end{corollary}

\begin{proof}
Only the last claim needs verification, so suppose that $\mcI \subseteq {^{\perp}}(a^{\perp})$ is an ideal containing $a$ that is closed under coproducts of morphisms, and $\ME$-extensions by projective morphisms. Then $a^{(I)} \in \mcI,$ for every index set $I,$ which implies that $\mcI$ is an $a^{\perp}$-cosyzygy ideal, and therefore that $\mcI = \mcEproj \diamond \mcI = {^{\perp}}(a^{\perp}).$
\end{proof}

Similarly, the ideal cotorsion pair generated by a set $\mcM$ of morphisms is complete. This is the ideal version of a result ~\cite[Theorem 10]{ET} of Eklof and Trlifaj.

\begin{corollary} \label{C:small left closure}
Suppose that $(\mcA; \mcE)$ has enough projective morphisms and injective morphisms. An ideal $\mcI \subseteq \mcA$ is the left closure of a set $\mcM = \{ m^j \: | \: j \in J \}$ of morphisms if and only if it is the least ideal of $\mcA$ that {\rm ($1$)} contains $\mcM;$ {\rm ($2$)} is closed under coproducts; and
{\rm ($3$)} is closed under $\ME$-extensions of projective morphisms. Every such ideal is special precovering.
\end{corollary}

\begin{proof}
Apply Corollary~\ref{C:cyclic left closure} to the morphism $a = \amalg_j \, m^j$ and recall that the left closure of $\mcM$ is the left closure of $a.$
\end{proof}

\subsection{The object version} Let us consider the special case of the BET Lemma when the arrow $a = 1_A$ is an object.

\begin{corollary} \label{C:ET Lemma}
Let $A$ and $B$ be objects of $(\mcA; \mcE)$ and suppose that the $\Ext$-group $\Ext_{\mcA} (A,B)$ is a set. A special preenvelope of $B$ with respect to the ideal $A^{\perp} := \{ b : B_0 \to B_1 \; | \; \Ext_{\mcA} (A,b) = 0  \}$ is given by the pushout in the bottom row of
$$\xymatrix@C=30pt@R=30pt{
B \ar[r] \ar@{=}[d] & C_{\xi} \ar[r]  \ar[d] & A \ar[d]^{\iota_{\xi}} \\
B \ar[r] & \amalg_B \, C_{\xi} \ar[r] & A^{(I)},
}$$
indexed by $\xi \in I = \Ext (A,B).$ If $(\mcA; \mcE)$ has enough projective morphisms and injective morphisms, then the left closure ${^{\perp}}(A^{\perp}) = \mcEproj \diamond \langle \Add (A) \rangle$ is a special precovering ideal.
\end{corollary}

\begin{proof}
Just recall that $\Ext_{\mcA}(A,B) = \Ext_{\Arr (\mcA)}(1_A, 1_B),$ while $\Ext_{\mcA} (A,b) = \Ext_{\mcA} (1_A,b)$ and note that the object arrows $1_A$ are closed under coproducts and pushouts over object arrows $1_B$ in $(\Arr (\mcA); \Arr (\mcE)).$ Then apply the BET Lemma and Corollary~\ref{C:cyclic left closure}.
\end{proof}

If the category $\RMod$ of left $R$-modules is endowed with the standard exact structure of an abelian category and $A = \oplus \; \{ V \; | \; V \in \Rmod \}$ is the direct sum of all the finitely presented left $R$-modules, then the morphisms of $A^{\perp}$ are called {\em cophantom} morphisms~\cite[Proposition 37]{FGHT} and the ideal $A^{\perp}$ is denoted by $\Psi.$ Corollary~\ref{C:ET Lemma} gives another proof of the completeness ({\em ibid.}) of the ideal cotorsion pair $({^{\perp}}\Psi, \Psi),$ and shows that the ideal ${^{\perp}}\Psi$ is the object ideal generated by the modules ${_R}M$ that arise as extensions
$$\xymatrix@C=30pt{0 \ar[r] & P \ar[r] & M \ar[r] & N \ar[r] & 0
}$$
where $P$ is projective and $N$ is pure projective.

\end{document}